\documentclass[10pt,leqno]{article}

\usepackage{amsmath,amssymb,amsthm,mathrsfs,dsfont}

\usepackage[margin=3cm]{geometry} %¿í°æÊ±ÓÃ£¬ÅäºÏË«±¶ÐÐ¾à

\usepackage{titlesec,hyperref}

\usepackage{color}
\usepackage{fancyhdr}
\titleformat{\subsection}{\it}{\thesubsection.\enspace}{1pt}{}

\newtheorem{theo}{Theorem}[section]
\newtheorem{lemm}[theo]{Lemma}

\newtheorem{prop}[theo]{Proposition}

\numberwithin{equation}{section}

\allowdisplaybreaks %ÔÊÐí¹«Ê½·ÖÒ³ÏÔÊ¾

\begin{document}
\title{Persistence property and the local well-posedness of the modified Camassa-Holm equation in critical Besov equation
	\hspace{-4mm}
}

\author{Zhen $\mbox{He}^1$ \footnote{Email: hezh56@mail2.sysu.edu.cn},\quad
	Zhaoyang $\mbox{Yin}^{1,2}$\footnote{E-mail: mcsyzy@mail.sysu.edu.cn}\\
$^1\mbox{Department}$ of Mathematics,
Sun Yat-sen University, Guangzhou 510275, China\\
$^2\mbox{School}$ of Science,\\ Shenzhen Campus of Sun Yat-sen University, Shenzhen 518107, China}

\date{}
\maketitle
\hrule

\begin{abstract}
	In this paper, we first establish the local well-posednesss for the Cauchy problem of a modified Camassa-Holm (MOCH) equation in critical Besov spaces $B^{\frac 1 p}_{p,1}$ with $1\leq p<+\infty.$ The obtained results  improve considerably the recent result in \cite{Luo1}. Then we show the persiscence property of MOCH.
	
\vspace*{5pt}
	
	\vspace*{5pt}
\end{abstract}

\vspace*{10pt}

%\phantomsection
%\addcontentsline{toc}{section}{\contentsname}
%添加目录到书签
\tableofcontents

	\section{Introduction}
	
	  ~~In this paper,  we consider the Cauchy problem of the following modified Camassa-Holm (MOCH) equation \cite{2021}
	\begin{equation}\label{eq0}
		\left\{\begin{array}{l}
			\gamma_t=\lambda(v_x-\gamma-\frac{\lambda}{1}v\gamma)_x
			,  \quad t>0,\ x\in\mathbb{R},  \\
			v_{xx}=\gamma_x+\frac {\gamma^2} {2\lambda} ,  \quad t\geq0,\ x\in\mathbb{R},  \\
			\gamma(0,x)=\gamma_0(x),  \quad x\in\mathbb{R},
		\end{array}\right.
	\end{equation}
	which was called by Gorka and Reyes\cite{GR2010}. Let $G={\partial_x}^2-1,m=Gv$
	
	The equation \eqref{eq0} can be rewritten as
	\begin{equation}\label{eq1}
		\left\{\begin{array}{l}
			\gamma_t+G^{-1}m\gamma_x=\frac {\gamma^2}{2} +\lambda G^{-1}m-
			\gamma G^{-1}m_x,  \quad t>0,\ x\in\mathbb{R},  \\
			m=\gamma_x+\frac {\gamma^2} {2\lambda} ,  \quad t\geq0,\ x\in\mathbb{R},  \\
			\gamma(0,x)=\gamma_0(x),  \quad x\in\mathbb{R}.
		\end{array}\right.
	\end{equation}

	  The equation $\eqref{eq0}$ was first studied through the geometric approach in\cite{Chern} \cite{Reyes}. Conservation laws and the existence and uniqueness of weak solutions to the
	modified Camassa–Holm equation were presented in \cite{GR2010}. We observe that if we solve
	\eqref{eq1}, then $m$ will formally satisfy the following physical form of the Camassa–Holm

	\begin{align}
		m_t=-2vm_x-mv_x+\lambda v_x
	\end{align}

	  If $\lambda=0$, it is known as the well-known Camassa–Holm (CH) equation. 
	As far as we know, the CH equation has many properties, such as: integrability  \cite{Constantin2001,Constantin2006,Constantin1999},  Hamiltonian structure, infinitely many conservation laws \cite{Fuchssteiner1981,Constantin2001}. The local well-posedness  of  CH equation  has been proved in  Sobolev spaces $H^s, s>\frac 3 2$ and in Besov space $B^{s}_{p,r}$ with $ s>\max\{\frac{3}{2},1+\frac 1 p \}$ or $s=1+\frac 1 p,p\in [1,2],r=1$ in \cite{Constantin1998,Constantin1998w-l,Danchin2001inte,Danchin2003wp,Liu2019ill,Himonas2012CN,Li2016nwpC}. In \cite{Constantin2000gw,Xin2000ws}, the authors showed  the existence and uniqueness of global weak solutions for the CH equation. In addition, Bressan and Constantin studied  the existence of the global conservative solutions \cite{book}  and global disspative solutions \cite{Bressan2007gd}  in $H^1(\mathbb{R}).$  Recently, Ye, Yin and Guo \cite{2021} gave the local well-posedness in $B^{1+\frac 1 p}_{p,1}$ with $p\in (1,+\infty].$ For the ill-posed problems,  the author \cite{Byers2006illhs} proved that CH equation is ill-posed in $H^{s}(\mathbb{R}), s<\frac 3 2.$ Danchin  proved that the CH equation is ill-posed in $B^{\frac 3 2}_{p,\infty} $  \cite{Danchin2001inte,Danchin2003wp}  and Guo et al. discussed  the equation for the shallow water wave type is ill-posed in $H^{\frac 3 2}$ and  $B^{1+\frac 1 p}_{p,r}$ with $ p\in [1,+\infty], r\in (1,+\infty]$   \cite{Liu2019ill}. In \cite{Li2021}, Li, Yu et.al  showed the ill-posedness of the Camassa-Holm equation in $B^{s}_{p,\infty}$ for $s>2+\max\{\frac 3 2, 1+\frac 1  p\}, 1\leq p\leq +\infty.$
	
	  Luo,Qiao and Yin studied the locally well-posedness in $B^s_{p,r},s\textgreater \max\{\frac{1}{2},\frac{1}{p}\}$ or $s=\frac{1}{p},1\leq p\leq 2,r=1$, blow up phenomena, global existence for periodic MOCH and global conservative solution\cite{Luo1}\cite{Luo2}
	\cite{Luo3}.
	
	  This paper is aimed to complete the locally well-posedness in critical Besov space $B^{\frac 1 p}_{p,1}$ with $1\leq p<+\infty$. For existence, we use the compactness method. For uniqueness, due to the defect of remainder operator in Bony decompositon, we use the characteristic line to overcome this diffculty. Although the regularity is still a hard problem to deal with, the good structure of the transport term provide us a way to lift regularity by using operator $\partial_x-1$. We will transform\eqref{eq1} to a better form about n to obtain the uniqueness of $n$, then the uniqueness of $\gamma$ can be deduced from the uniqueness of $n$. For the continuous dependence, the first two steps provide us with great convenience to the problem. And in the last part of the paper, we will talk about the persiscence property. 
	
	\section{\textbf{Preliminaries}}
	
	In this section, we will present some propositions about the Littlewood-Paley decomposition and Besov spaces.
	\begin{prop}\cite{Chemin2011,He2017}
		Let $s\in\mathbb{R},\ 1\leq p,p_1,p_2,r,r_1,r_2\leq\infty.$  \\
		(1) $B^s_{p,r}$ is a Banach space, and is continuously embedded in $\mathcal{S}'$. \\
		(2) If $r<\infty$, then $\lim\limits_{j\rightarrow\infty}\|S_j u-u\|_{B^s_{p,r}}=0$. If $p,r<\infty$, then $C_0^{\infty}$ is dense in $B^s_{p,r}$. \\
		(3) If $p_1\leq p_2$ and $r_1\leq r_2$, then $ B^s_{p_1,r_1}\hookrightarrow B^{s-d(\frac 1 {p_1}-\frac 1 {p_2})}_{p_2,r_2}. $
		If $s_1<s_2$, then the embedding $B^{s_2}_{p,r_2}\hookrightarrow B^{s_1}_{p,r_1}$ is locally compact. \\
		(4) $B^s_{p,r}\hookrightarrow L^{\infty} \Leftrightarrow s>\frac d p\ \text{or}\ s=\frac d p,\ r=1.
		\quad $ \\
		(5) Fatou property: if $(u_n)_{n\in\mathbb{N}}$ is a bounded sequence in $B^s_{p,r}$, then an element $u\in B^s_{p,r}$ and a subsequence $(u_{n_k})_{k\in\mathbb{N}}$ exist such that
		$$ \lim_{k\rightarrow\infty}u_{n_k}=u\ \text{in}\ \mathcal{S}'\quad \text{and}\quad \|u\|_{B^s_{p,r}}\leq C\liminf_{k\rightarrow\infty}\|u_{n_k}\|_{B^s_{p,r}}. $$
		(6) Let $m\in\mathbb{R}$ and $f$ be a $S^m$- multiplier, (i.e. f is a smooth function and satisfies that $\forall\alpha\in\mathbb{N}^d$, $\exists C=C(\alpha)$, such that $|\partial^{\alpha}f(\xi)|\leq C(1+|\xi|)^{m-|\alpha|},\ \forall\xi\in\mathbb{R}^d)$.
		Then the operator $f(D)=\mathcal{F}^{-1}(f\mathcal{F}\cdot)$ is continuous from $B^s_{p,r}$ to $B^{s-m}_{p,r}$.
	\end{prop}
	We give two useful interpolation inequalities.
	\begin{prop}\label{prop}\cite{Chemin2011,He2017}
		(1) If $s_1<s_2$, $\theta \in (0,1)$, and $(p,r)$ is in $[1,\infty]^2$, then we have
		$$ \|u\|_{B^{\theta s_1+(1-\theta)s_2}_{p,r}}\leq \|u\|_{B^{s_1}_{p,r}}^{\theta}\|u\|_{B^{s_2}_{p,r}}^{1-\theta}. $$
		(2) If $s\in\mathbb{R},\ 1\leq p\leq\infty,\ \varepsilon>0$, a constant $C=C(\varepsilon)$ exists such that
		$$ \|u\|_{B^s_{p,1}}\leq C\|u\|_{B^s_{p,\infty}}\ln\Big(e+\frac {\|u\|_{B^{s+\varepsilon}_{p,\infty}}}{\|u\|_{B^s_{p,\infty}}}\Big). $$
	\end{prop}
	\begin{prop}\cite{Chemin2011}
		Let $s\in\mathbb{R},\ 1\leq p,r\leq\infty.$
		\begin{equation*}\left\{
			\begin{array}{l}
				B^s_{p,r}\times B^{-s}_{p',r'}\longrightarrow\mathbb{R},  \\
				(u,\phi)\longmapsto \sum\limits_{|j-j'|\leq 1}\langle \Delta_j u,\Delta_{j'}\phi\rangle,
			\end{array}\right.
		\end{equation*}
		defines a continuous bilinear functional on $B^s_{p,r}\times B^{-s}_{p',r'}$. Denoted by $Q^{-s}_{p',r'}$ the set of functions $\phi$ in $\mathcal{S}'$ such that
		$\|\phi\|_{B^{-s}_{p',r'}}\leq 1$. If $u$ is in $\mathcal{S}'$, then we have
		$$\|u\|_{B^s_{p,r}}\leq C\sup_{\phi\in Q^{-s}_{p',r'}}\langle u,\phi\rangle.$$
	\end{prop}
	We then have the following product laws:
	\begin{lemm}\label{product}\cite{Chemin2011,He2017}
		(1) For any $s>0$ and any $(p,r)$ in $[1,\infty]^2$, the space $L^{\infty} \cap B^s_{p,r}$ is an algebra, and a constant $C=C(s,d)$ exists such that
		$$ \|uv\|_{B^s_{p,r}}\leq C(\|u\|_{L^{\infty}}\|v\|_{B^s_{p,r}}+\|u\|_{B^s_{p,r}}\|v\|_{L^{\infty}}). $$
		(2) If $1\leq p,r\leq \infty,\ s_1\leq s_2,\ s_2>\frac{d}{p} (s_2 \geq \frac{d}{p}\ \text{if}\ r=1)$ and $s_1+s_2>\max(0, \frac{2d}{p}-d)$, there exists $C=C(s_1,s_2,p,r,d)$ such that
		$$ \|uv\|_{B^{s_1}_{p,r}}\leq C\|u\|_{B^{s_1}_{p,r}}\|v\|_{B^{s_2}_{p,r}}. $$
		(3) If $1\leq p\leq 2$,  there exists $C=C(p,d)$ such that
		$$ \|uv\|_{B^{\frac d p-d}_{p,\infty}}\leq C \|u\|_{B^{\frac d p-d}_{p,\infty}}\|v\|_{B^{\frac d p}_{p,1}}. $$
	\end{lemm}
	The Gronwall lemma as follows.
	\begin{lemm}\label{osgood}\cite{Chemin2011}
		Let $f(t),~ g(t)\in C^{1}([0,T]), f(t), g(t)>0.$ Let $h(t)$ is a continuous function on $[0,T].$ Assume that, for any $t\in [0,T]$ such that
		$$\frac 1 2 \frac{d}{dt}f^{2}(t)\leq h(t)f^{2}(t)+g(t)f(t).$$
		Then for any time $t\in [0,T],$ we have
		$$f(t)\leq f(0)exp\int_0^th(\tau)d\tau+\int_0^t g(\tau)\ exp(\int_\tau ^t h(\tau)dt')d\tau.$$
	\end{lemm}
	Now we state some useful results in the transport equation theory, which are important to the proofs of our main theorem later.
	\begin{equation}\label{transport}
		\left\{\begin{array}{l}
			f_t+v\cdot\nabla f=g,\ x\in\mathbb{R}^d,\ t>0, \\
			f(0,x)=f_0(x).
		\end{array}\right.
	\end{equation}
	\begin{lemm}\label{existence}\cite{Chemin2011}
		Let $1\leq p\leq p_1\leq\infty,\ 1\leq r\leq\infty,\ s> -d\min(\frac 1 {p_1}, \frac 1 {p'})$. Let $f_0\in B^s_{p,r}$, $g\in L^1([0,T];B^s_{p,r})$, and let $v$ be a time-dependent vector field such that $v\in L^\rho([0,T];B^{-M}_{\infty,\infty})$ for some $\rho>1$ and $M>0$, and
		$$
		\begin{array}{ll}
			\nabla v\in L^1([0,T];B^{\frac d {p_1}}_{p_1,\infty}), &\ \text{if}\ s<1+\frac d {p_1}, \\
			\nabla v\in L^1([0,T];B^{s-1}_{p,r}), &\ \text{if}\ s>1+\frac d {p_1}\ or\ (s=1+\frac d {p_1}\ and\ r=1).
		\end{array}
		$$
		Then the equation \eqref{transport} has a unique solution $f$ in \\
		-the space $C([0,T];B^s_{p,r})$, if $r<\infty$; \\
		-the space $\Big(\bigcap_{s'<s}C([0,T];B^{s'}_{p,\infty})\Big)\bigcap C_w([0,T];B^s_{p,\infty})$, if $r=\infty$.
	\end{lemm}
	\begin{lemm}\label{priori estimate}\cite{Chemin2011,Li2016nwpC}
		Let $s\in\mathbb{R},\ 1\leq p,r\leq\infty$.
		There exists a constant $C$ such that for all solutions $f\in L^{\infty}([0,T];B^s_{p,r})$ of \eqref{transport} in one dimension with initial data $f_0\in B^s_{p,r}$, and $g\in L^1([0,T];B^s_{p,r})$, we have for a.e. $t\in[0,T]$,
		$$ \|f(t)\|_{B^s_{p,r}}\leq \|f_0\|_{B^s_{p,r}}+\int_0^t \|g(t')\|_{B^s_{p,r}}dt'+\int_0^t V^{'} (t^{'})\|f(t)\|_{B^s_{p,r}}dt{'} $$
		or
		$$ \|f(t)\|_{B^s_{p,r}}\leq e^{CV(t)}\Big(\|f_0\|_{B^s_{p,r}}+\int_0^t e^{-CV(t')}\|g(t')\|_{B^s_{p,r}}dt'\Big) $$
		with
		\begin{equation*}
			V'(t)=\left\{\begin{array}{ll}
				\|\nabla v\|_{B^{s+1}_{p,r}},\ &\text{if}\ s>\max(-\frac 1 2,\frac 1 {p}-1), \\
				\|\nabla v\|_{B^{s}_{p,r}},\ &\text{if}\ s>\frac 1 {p}\ \text{or}\ (s=\frac 1 {p},\ p<\infty, \ r=1),
			\end{array}\right.
		\end{equation*}
		and when $s=\frac 1 p-1,\ 1\leq p\leq 2,\ r=\infty,\ \text{and}\ V'(t)=\|\nabla v\|_{B^{\frac 1 p}_{p,1}}$.\\
		If $f=v,$ for all $s>0, V^{'}(t)=\|\nabla v(t)\|_{L^{\infty}}.$
	\end{lemm}

	\begin{lemm}\label{cont1}\cite{Chemin2011,Li2016nwpC}
		Let $y_0\in B^{\frac{1}{p}}_{p,1}$ with $1\leq p<\infty,$  and $f\in L^1([0,T];B^{\frac{1}{p}}_{p,1}).$ Define $\bar{\mathbb{N}}=\mathbb{N} \cup \{\infty\},$ for $n\in \bar{\mathbb{N}},$ and by $y_n \in C([0,T];B^{\frac{1}{p}}_{p,1})$ the solution of
		\begin{equation}
			\left\{\begin{aligned}
				&\partial_ty_n+A_n(u)\partial_xy_n=f,\quad x\in \mathbb{R},\\
				&y_n(t,x)|_{t=0}=y_0(x). \\
			\end{aligned} \right. \label{e1}
		\end{equation}
		Assume for some $\alpha(t)\in L^1(0,T),\  \sup\limits_{n\in \bar{\mathbb{N}}} \|A_n(u)\|_{B^{1+\frac{1}{p}}_{p,1}}\leq \alpha (t).$ If $A_n(u)$ converges in $A_{\infty}(u)$ in $L^1([0,T];B^{\frac{1}{p}}_{p,1}),$ then the sequence $(y_n)_{n\in \mathbb{N}}$ converges in $C([0,T];B^{\frac{1}{p}}_{p,1}).$
	\end{lemm}
	Let us consider the following initial value problem
	\begin{equation}\label{eq101}
		\left\{\begin{array}{l}
			y_t=u(t,y),t\in [0,T)  \\
			
			y(0,x)=x,x\in\mathbb{R}.
		\end{array}\right.
	\end{equation}
	\begin{lemm}\cite{Conper,Yinper}
		 Let u $\in C( [ 0,T);H^s)\cap  C^1([0,T);H^{s-1}),s\geq 2$ .Then\eqref{eq101} has a unique solution$ y \in C^1([0,T)\times \mathbb{R};\mathbb{R})$. Moreover, the map $y(t,\cdot)$ is an increasing diffeomorphism of$ \mathbb{R}$ with  
		$$y_x(t,x)=exp(\int_{0}^{t} u_x(s,q(s,x))ds)\textgreater 0,\forall(t,x)\in [0,T)\times \mathbb{R}$$ 
	\end{lemm}
	\begin{lemm}\cite{Luo1}
		Let $1\leq p,r\leq \infty,s\in \mathbb{R}$ and let (s,p,r) satisfy the condition $s\textgreater max(\frac{1}{2},\frac{1}{p})$ or $s=\frac{1}{p},1\leq p\leq 2,r=1.$ Assume that $ \gamma_0\in B^s_{p,r}$. Then there
		exists a time $T \textgreater 0$ such that $\eqref{eq1}$ has a unique solution $\gamma$ in $E^s_{p,r}(T)$. Moreover the solution depends continuously on the initial data.
	\end{lemm}
   \begin{lemm}\label{2continuity}\cite{Li2016nwpC}
   	Suppose that ~$1\leq p\leq\infty,\ 1\leq r<\infty,\ s>\frac d p$\ (or \ $s=\frac d p,\ p<\infty,\ r=1)$. Let ~$\bar{\mathbb{N}}=\mathbb{N}\cup\{\infty\}$. Asuming that ~$(v^n)_{n\in\bar{\mathbb{N}}}\in C([0,T];B^{s+1}_{p,r})$, and ~$(f^n)_{n\in\bar{\mathbb{N}}} \in C([0,T];B^s_{p,r})$ solves the equation:
   	\begin{equation}
   		\left\{\begin{array}{l}
   			f^n_t+v^n\cdot\nabla f^n=g,\ x\in\mathbb{R}^d,\ t>0, \\
   			f^n(0,x)=f_0(x)
   		\end{array}\right.
   	\end{equation}
   	with ~$f_0\in B^s_{p,r},\ g\in L^1([0,T];B^s_{p,r})$, then there exists ~$\alpha\in L^1([0,T])$,  such that
   	$$\sup\limits_{n\in\bar{\mathbb{N}}}\|v^n(t)\|_{B^{s+1}_{p,r}}\leq \alpha(t).$$
   	If ~$v^n$ converges to ~$v^{\infty}$ in ~$L^1([0,T];B^s_{p,r})$ , then ~$f^n$ will converge to ~$f^{\infty}$ in ~$C([0,T];B^s_{p,r})$ .
   \end{lemm}
\section{Local well-posedness}
In this section, we use Theorem 1.1 in \cite{2021}  to obtain the well-posedness of equation \eqref{eq1} in Besov spaces $B^{\frac 1 p}_{p,1}$ with $ 1\leq p <+\infty.$ Our main theorem can be stated as follows.
\begin{theo}\label{the1}
	Let $\gamma^0$ be in $B^{\frac{1}{p}}_{p,1}$ with $p\in [1,\infty)$ .
	Then there exists a time $T>0$ such that \eqref{eq1} has a unique solution $\gamma$ in $E^p_T\triangleq C\Big([0,T];B^{\frac 1 p}_{p,1}\Big)\cap C^1\Big([0,T];B^{\frac{1}{p}-1}_{p,1}\Big).$  Moreover,  the solution depends continuously on the initial data.
\end{theo}
\begin{proof}\
	Now  we prove Theorem \ref{the1} in three steps.
	
	\textbf{Step 1: Existence.}
	
	We firstly set ~$\gamma^0\triangleq0$ and define a sequences~$(\gamma^n)_{n\in\mathbb{N}}$
	of smooth functions by solving the following linear transport equations:
	\begin{equation}\label{14}
		\left\{\begin{array}{l}
			\gamma_t^{n+1}+G^{-1}m^{n}\gamma_x^{n+1}=\frac {(\gamma^{n})^2}{2}\ +\lambda G^{-1}m^{n}-\gamma G^{-1}m^{n}_x,  \\
			m^{n}=\gamma^{n}_x+\frac {(\gamma^{n})^2} {2\lambda}\ ,  \\
			\gamma^{n+1}(0,x)=S_{n+1}\gamma_0.
		\end{array}\right.
	\end{equation}
	we denote
	~$G^n=G^{-1}m^{n},\ F^n=\frac {(\gamma^{n})^2}{2}\ +\lambda G^{-1}m^{n}-\gamma G^{-1}m^{n}_x.$ We assume that $\gamma _n\in L^\infty([0,T];B^{\frac{1}{p}}_{p,1})$ for all $T>0,$ it follows that
	\begin{align}
		\|G_x^n\|_{B^{s}_{p,r}}&\leq C\|m^n\|_{B^{s-1}_{p,r}}
		\notag  \\
		&\leq C(\|\gamma^n\|_{B^{s}_{p,r}}+
		\|\gamma^n\|^2_{B^{s}_{p,r}}), \label{2ineq1}  \\
		\|F^n\|_{B^{s}_{p,r}}&\leq  C(\|\gamma^n\|^2_{B^s_{p,r}}+\|G^{-1}m^n\|_{B^s_{p,r}}+\|\gamma^n\|_{B^s_{p,r}}\|G^{-1}m_x^n\|_{B^s_{p,r}}) \notag  \\
		&\leq C(\|\gamma^n\|_{B^{s}_{p,r}}
		+\|\gamma^n\|^2_{B^{s}_{p,r}}+\|\gamma^n\|^3_{B^{s}_{p,r}}). \label{2ineq2}
	\end{align}
	Combining Lemma \ref{existence}  with \eqref{14},
	we deduce that there exists a global solution $\gamma _{n+1}\in E^{p}_T=C([0,T];{B^{\frac{1}{p}}_{p,1}})\cap
	C^1([0,T];{B^{\frac{1}{p}-1}_{p,1}}).$ Making use of  Lemma \ref{priori estimate} yields
	\ we define $R_n=\|\gamma^{n}(t)\|_{B^{s}_{p,r}}.$
	\begin{align}
		R_{n+1} &\leq e^{C\int_0^t \|G_x^n\|_{B^{s}_{p,r}}dt'}
		\Big(\|S_{n+1}\gamma_0\|_{B^{s}_{p,r}}+\int_0^t e^{-C\int_0^{t'} \|G_x^n\|_{B^{s}_{p,r}} dt''}
		\|F^n\|_{B^{s}_{p,r}}dt'\Big)  \notag\\
		&\leq Ce^{C\int_0^t R_n+R_n^2 dt'}\Big(\|\gamma_0\|_{B^{s}_{p,r}}+\int_0^t e^{-C\int_0^{t'} R_n+R_n^2 dt''} (R_n+R_n^2+R_n^3)dt'\Big).  \label{18}
	\end{align}
	For fixed  $T>0$ such that $ 4C^3 T\|\gamma_0\|^2_{B^{s}_{p,r}} <1 $ and suppose that
	\begin{equation}\label{19}
		\forall t\in [0,T],\ R_n\leq
		\frac{C\|\gamma_0\|_{B^{s}_{p,r}}}{2(1-4C^3 t\|\gamma_0\|^2_{B^{s}_{p,r}})^{\frac 12}}.
	\end{equation}
	Plugging \eqref{19} into \eqref{18}, we have
	\begin{align*}
		R_{n+1} &\leq
		C\|\gamma_0\|_{B^{s}_{p,r}}(1-4C^3 t\|\gamma_0\|^2_{B^{s}_{p,r}})^{-\frac 14}
		\Big(1+C^3\|\gamma_0\|^2_{B^{s}_{p,r}} \int_0^t (1-4C^3 t'\|\gamma_0\|^2_{B^{s}_{p,r}})^{-\frac 54} dt'\Big)  \\
		&\leq \frac{C\|\gamma_0\|_{B^{s}_{p,r}}}{2(1-4C^3 t\|\gamma_0\|^2_{B^{s}_{p,r}})^{\frac 12}}.
	\end{align*}
	Therefore, the sequences  $(\gamma_n)_{n\in \mathbb{N}}$ is uniformly  bounded in $L^{\infty}([0,T];{B^{\frac{1}{p}}_{p,1}})$.
	
	Making use of  the compactness method for the approximating sequence $(\gamma_n)_{n\in\mathbb{N}}$ to get a solution $\gamma$ of \eqref{eq1}.  Owing to the uniformly boundedness of $\gamma_n$ in $L^{\infty}([0,T];B^{\frac{1}{p}}_{p,1}),$ we get  $\partial_t\gamma_n$ is uniformly bounded in $L^{\infty}([0,T];B^{\frac{1}{p}-1}_{p,1}).$ Therefore,
	$\gamma_n$ is~uniformly~bounded~in~  $C([0,T];B^{\frac{1}{p}}_{p,1})\cap C^{\frac 1 2}([0,T];B^{\frac{1}{p}-1}_{p,1}).$
	We give a sequence  $(\phi_j)_{j\in\mathbb{N}}$  of smooth functions with values in $[0,1],$ supported in $B(0,j+1)$ and equals to   $1$ on $B(0,j).$ Now, taking advantage of  Theorem 2.94 in  \cite{Chemin2011}, it is easy to check that the map $\gamma_n\mapsto \phi_j \gamma_n$ is compact from $B^{\frac{1}{p}}_{p,1}$ to $B^{\frac{1}{p}-1}_{p,1}$. Hence, by using Ascoli's theorem, we can see there exists some function $\gamma_j$ such that, up to extraction, the sequence $(\phi_j \gamma_n)_{j\in\mathbb{N}}$ converges to  $\gamma_j.$  On the  other hand,
	using the Cantor diagonal process, we deduce that there exists a subsequence  of $(\gamma_j)_{j\in\mathbb{N}}$ ( still mark it as $(\gamma_j)_{j\in\mathbb{N}}$ ) such that for any $j\in \mathbb{N},$ $\phi_j \gamma_n$  converges to $\gamma_j$ in $C([0,T]; B^{\frac 1 p -1}_{p,1}).$ Note that $\phi_j\phi_{j+1}=\phi_j,$  we thus have  $\gamma_j=\phi_j \gamma_{j+1}.$ Consequently, for  any $\phi\in\mathcal{D},$ we deduce that there exists some function $\gamma$  such that  the sequence  $(\phi \gamma_n)_{n\geq 1}$   tends (up to a subsequence
	independent of $\phi$) to $                                                                                                                                                                                                                                                                                                                                                                                                                                                                                                                                                                                                                                                                                                                                                                                                                                                                                                                                                                                                                                                                                                                                                                                                                                                                                                                                                                                                                                                                                                                                                                                                                                                                                                                                                                                                                                                                                                                                                                                                                                                                                                                                                                                                                                                                                                                                                                                                                                                                                                                                                                                                                                                                                                                                                                                                                                                                                                                                                                                                                                                                                                                                                                                                                                                                                                                                                                                                                                                                                                                                                                                                                                                                                                                                                                                                                                                                                                                                                                                                                                                                                                                                                                                                                                                                                                                                                                                                                                                                                                                                                                                                                                                                                                                                                                                                                                                                                                                                                                                                                                                                                                                                                                                                                                                                                                                                                                                                                                                                                                                                                                                                                                                                                                                                                                                                                                                                                                                                                                                                                                                                                                                                                                                                                                                                                                                                                                                                                                                                                                                                                                                                                                                                                                                                                                                                                                                                                                                                                                                                                                                                                                                                                                                                                                                                                                                                                                                                                                                                                                                                                                                                                                                                                                                                                                                                                                                                                                                                                                                                                                                                                                                                                                                                                                                                                                                                                                                                                                                                                                                                                                                                                                                                                                                                                                                                                                                                                                                                                                                                                                                                                                                                                                                                                                                                                                                                                                                                                                                                                                                                                                                                                                                                                                                                                                                                                                                                                                                                                                                                                                                                                                                                                                                                                                                                                                                                                                                                                                                                                                                                                                                                                                                                                                                                                                                                                                                                                                                                                                                                                                                                                                                                                                                                                                                                                                                                                                                                                                                                                                                                                                                                                                                                                                                                                                                                                                                                                                                                                                                                                                                                                                                                                                                                                                                                                                                                                                                                                                                                                                                                                                                                                                                                                                                                                                                                                                                                                                                                                                                                                                                                                                                                                                                                                                                                                                                                                                                                                                                                                                                                                                                                                                                                                                                                                                                                                                                                                                                                                                                                                                                                                                                                                                                                                                                                                                                                                                                                                                                                                                                                                                                                                                                                                                                                                                                                                                                                                                                                                                                                                                                                                                                                                                                                                                                                                                                                                                                                                                                                                                                                                                                                                                                                                                                                                                                                                                                                                                                                                                                                                                                                                                                                                                                                                                                                                                                                                                                                                                                                                                                                                                                                                                                                                                                                                                                                                                                                                                                                                                                                                                                                                                                                                                                                                                                                                                                                                                                                                                                                                                                                                                                                                                                                                                                                                                                                                                                                                                                                                                                                                                                                                                                                                                                                                                                                                                                                                                                                                                                                                                                                                                                                                                                                                                                                                                                                                                                                                                                                                                                                                                                                                                                                                                                                                                                                                                                                                                                                                                                                                                                                                                                                                                                                                                                                                                                                                                                                                                                                                                                                                                                                                                                                                                                                                                                                                                                                                                                                                                                                                                                                                                                                                                                                                                                                                                                                                                                                                                                                                                                                                                                                                                                                                                                                                                                                                                                                                                                                                                                                                                                                                                                                                                                                                                                                                                                                                                                                                                                                                                                                                                                                                                                                                                                                                                                                                                                                                                                                                                                                                                                                                                                                                                                                                                                                                                                                                                                                                                                                                                                                                                                                                                                                                                                                                                                                                                                                                                                                                                                                                                                                                                                                                                                                                                                                                                                                                                                                                                                                                                                                                                                                                                                                                                                                                                                                                                                                                                                                                \phi \gamma$ in $C([0,T];B^{\frac{1}{p}-1}_{p,1}).$   In addition,
	using  the uniform boundness of $(\gamma^n)_{n\geq 1}$ and the Fatou property, we can find that $\gamma$ is bounded in $L^{\infty}([0,T];B^{\frac{1}{p}}_{p,1}).$ For any $\varepsilon>0,$  we see that  $\phi \gamma_n$ tends to $\gamma$ in $C([0,T];B^{\frac{1}{p}-\varepsilon}_{p,1})$ by applying interpolation.  Since $(1-\partial_x^2)^{-1}\partial_x$ is a good operator, which is a routine process to prove that  $u$  is the solution of $\eqref{eq1}.$ Finally, using the equation again, we see that $\partial_t\gamma \in C([0,T];B^{\frac{1}{p}-1}_{p,1}).$
	
	\textbf{Step 2: Uniqueness.}
	
	In this step, we prove the uniqueness of the solution in the Lagrangian coordinate.
	Now, let us consider a transformation~$n=(\partial_x+1)G^{-1}\gamma=(\partial_x-1)^{-1}\gamma$.Then we have  $\gamma=(\partial_x-1)n$,and therefore,equation \eqref{eq1} is changed to 
	\begin{align}\label{20}
		n_t+un_x &=nu-G^{-1}(un_x-un)-\partial_xG^{-1}(un_x-un)+\frac 1 2 (-n^2+G^{-1}n^2_{x}+\partial_xG^{-1}n^2_{x})
		\notag  \\
		&+\lambda\partial_xG^{-1}n+\frac 1 2 ((\partial_x-1)^{-1}G^{-1}n^2_{x}-G^{-1}n^2),
	\end{align}
	where~ $u=n-\partial_xG^{-1}n+G^{-1}n+\frac {1} {2\lambda}(G^{-1}n^2_{x}+G^{-1}n^2-\partial_xG^{-1}n^2).$
	For the initial data ,we have 
	\begin{align}\label{3ini}
		n(0,x)=\bar{n}(x)=(\partial_x-1)^{-1}\bar{\gamma}(x).
	\end{align}
	and for simplicity, we denote ~$Q=nu-G^{-1}(un_x-un)-\partial_xG^{-1}(un_x-un)+\frac 1 2 (-n^2+G^{-1}n^2_{x}+\partial_xG^{-1}n^2_{x})
	+\lambda\partial_xG^{-1}n+\frac 1 2 ((\partial_x-1)^{-1}G^{-1}n^2_{x}-G^{-1}n^2)$
	Let $n$ be the smooth solution.  Introducing a new variable $\xi\in \mathbb{R},$ we define the characteristic $y(t,\xi) $ as
	\begin{equation}
		\left\{\begin{aligned}
			&y_t(t,\xi)=u(t,y(t,\xi)),\quad x\in \mathbb{R},\\
			&y(t,\xi)|_{t=0}=\bar{y}(\xi). \\
		\end{aligned} \right. \label{02}
	\end{equation}
	Define $N(t,\xi)=n(t,y(t,\xi)),U(t,\xi)=u(t,y(t,\xi)),$ then $N_{\xi}(t,\xi)=n_x(t,y(t,\xi))y_{\xi}(t,\xi).$ Moreover, the function  $U(t,\xi)$  is a solution of
	\begin{align}\label{05}
		N_t(t,\xi)=n_t(t,y(t,\xi))+y_t(t,\xi)n_x(t,y(t,\xi))=\tilde{Q}(t,\xi),
	\end{align}
	where
	\begin{align} \label{04}
		\tilde{Q}(t,\xi)&\triangleq NU(t,\xi)-G^{-1}{(UN_x-UN)}(t,\xi)-\partial_xG^{-1}{(UN_x-UN)}(t,\xi)
		\notag \\
		&~~~+\frac{1}{2}{(-N^2+G^{-1}{N_x}^2+\partial_xG^{-1}{N_x}^2)}+\lambda\partial_xG^{-1}{N}+\frac{1}{2}((\partial_x-1)^{-1}G^{-1}{N_x}^2-G^{-1}N^2)
		\notag \\
		&=NU(t,\xi)+\int_{\mathbb{R}}e^{-|y(t,\xi)-x|}(u(t,x)n_x(t,x)-u(t,x)n(t,x))dx
		\notag \\
		&~~~-\int_{\mathbb{R}}{\rm sgn}(y(t,\xi)-x)e^{-|y(t,\xi)-x|}(u(t,x)n_x(t,x)-u(t,x)n(t,x))dx
		\notag \\
		&~~~+\lambda\int_{\mathbb{R}}{\rm sgn}(y(t,\xi)-x)e^{-|y(t,\xi)-x|}{n_x}^2(t,x)dx-\frac{1}{2}\big(-N^2-\int_{\mathbb{R}}e^{-|y(t,\xi)-x|}{n_x}^2(t,x)dx
		\notag \\
		&~~~+\int_{\mathbb{R}}{\rm sgn}(y(t,\xi)-x)e^{-|y(t,\xi)-x|}{n_x}^2(t,x)dx \big)
		-\frac{1}{2}\int_{\mathbb{R}}e^{-|y(t,\xi) -x|}{n}^2(t,x)dx
		\notag \\
		&~~~+\frac{1}{2}(\partial_x+1)\int_{\mathbb{R}}\int_{\mathbb{R}}e^{-|y(t,\xi)-z|}e^{-|z-x|}{n_x}^2(t,x)dx
	\end{align}
	We perform the change of variables $x=y(t,\eta).$$z=y(t,\zeta).$ Hence, \eqref{04} may be written   as
	\begin{align}
		&\tilde{Q}(t,\xi)\triangleq Q(t,y(t,\xi))=NU(t,\xi)+\int_{\mathbb{R}}e^{-|y(\xi,t)-y(t,\eta)|}(U(t,\eta)N_x(t,\eta)\frac{1}{y_\eta}-U(t,\eta)N(t,\eta))y_\eta d\eta
		\notag \\
		&-\int_{\mathbb{R}}{\rm sgn}(y(t,\xi)-y(t,\eta))e^{-|y(\xi)-y(t,\eta)|}(U(t,\eta)N_x(t,\eta)\frac{1}{y_\eta}-U(t,\eta)N(t,\eta))y_\eta d\eta
		\notag \\
		&+\lambda\int_{\mathbb{R}}{\rm sgn}(y(t,\xi)-y(t,\eta))e^{-|y(\xi)-y(t,\eta)|}{N_x}^2(t,\eta)\frac{1}{y_\eta}d\eta
		\notag \\
		&-\frac{1}{2}(-N^2-\int_{\mathbb{R}}e^{-|y(\xi)-y_\eta|}{N_x}^2(t,\eta)\frac{1}{y_\eta}d\eta+\int_{\mathbb{R}}{\rm sgn}(y(t,\xi)-y(t,\eta))e^{-|y(\xi)-y(t,\eta)|}{N_x}^2(t,\eta)\frac{1}{y_\eta}d\eta)
		\notag \\
		&-\frac{1}{2}\int_{\mathbb{R}}e^{-|y(\xi)-y(t,\eta)|}{N}^2(t,\eta)y_\eta d\eta
		+\frac{1}{2}(\partial_x+1)\int_{\mathbb{R}}\int_{\mathbb{R}}e^{-|y(\xi)-y(t,\zeta)|}e^{-|y(t,\zeta)-y(t,\eta)|}{N_x}^2(t,\eta)\frac{y_\zeta}{y_\eta}d\zeta d\eta\label{09}  .
	\end{align}
	Taking the derivative of the equation \eqref{04}  with respect to variable $\xi,$ we have
	\begin{align}\label{07}
		N_{t\xi}=\tilde{Q}_{\xi}(t,\xi)=&N_x y_\xi U+NU_x y_\xi-\partial_xG^{-1}(UN_x-UN y_\xi)-(UN_x-UN y_\xi)-G^{-1}(UN_x-UN) y_\xi
		\notag \\
		&+\frac{1}{2}(-2NN_x y_\xi+\partial_xG^{-1} {N_x}^2 y_\xi+{N_x}^2 \frac{1}{y_\xi})+\lambda Ny_\xi+\lambda G^{-1}N y_\xi
		\notag \\
		&+\frac{1}{2}(G^{-1}{N_x}^2 y_\xi+(\partial_x -1)^{-1}G^{-1}{N_x}^2 y_\xi -\partial_xG^{-1}N^2 y_\xi).
	\end{align}
	
	Similarly, we  deduce that
	\begin{align}\label{17}
		y_{t\xi}(t,\xi)=U_{\xi}(t,\xi).  
	\end{align}
	Notice that the solutions of \eqref{04} and \eqref{17} can be written as
	\begin{align}
		&y(t,\xi)=\xi-\int_0^tU(\tau,\xi)d\tau,\label{b0}\\
		&y_{\xi}(t,\xi)=1-\int_0^tU_{\xi}(\tau,\xi)d\tau.\label{b1}
	\end{align}
	
	On the one hand, since $\|n\|_{B^{1+\frac{1}{p}}_{p,1}}\leq C\|(\partial_x -1)n\|_{B^{\frac{1}{p}}_{p,1}}\leq C\|\gamma\|_{B^{\frac{1}{p}}_{p,1}}$ using the fact that $\gamma$ is uniformly bounded in $C_T(B^{1+\frac{1}{p}}_{p,1})$  and the embedding $B^{1+\frac{1}{p}}_{p,1}\hookrightarrow W^{1,p}\cap W^{ 1,\infty},$ we easily deduce that  $n$ is uniformly bounded in $  C_T(W^{1,p}\cap W^{ 1,\infty} ).$ And it suffices  to prove that $y_{\xi}$ is uniformly bounded in $L^{\infty}_T(L^{\infty})$.
	On the other hand, we see that $\frac{1}{2}\leq y_{\xi}\leq C_{n_0}$  for  sufficiently small $T>0$.  With no loss of generality, let $t$ be sufficiently small,  we use the continuous method otherwise. Combining 
	\eqref{04} and the boundedness of $n$ in  $  C_T(W^{1,p}\cap W^{1,\infty} ),$ we have
	\begin{align*}
		&\|N(t,\xi)\|^p_{L^p}=\int_{\mathbb{R}} |N(t,\xi)|^pd\xi=\int_{\mathbb{R}} |n(t,y(t,\xi))|^p\frac{1}{y_{\xi}}dy\leq \|n\|^p_{L^p}\|\frac{1}{y_{\xi}}\|_{L^{\infty}}\leq 2\|n\|^p_{L^p}\leq C;\\ 
		&\|N_{\xi}(t,\xi)\|^p_{L^p}=\int_{\mathbb{R}} |N_{\xi}(t,\xi)|^p|y_{\xi}|^{p-1}d\xi=\int_{\mathbb{R}} |n_x(t,y(t,\xi))|^p|y_{\xi}|^{p-1}dy\leq \|n_x\|^p_{L^p}\|{y_{\xi}}\|^{p-1}_{L^{\infty}}\leq C^{p-1}_{u_0}\|n_x\|^p_{L^p}\leq C;\\
	\end{align*}
	From the above inequalities, we get $N(t,\xi) \in L^{\infty}_T(W^{1,p}\cap W^{1,\infty}), y(t,\xi)-\xi\in L^{\infty}_T(W^{1,p}\cap W^{1,\infty})$ and $\frac{1}{2}\leq y_{\xi}\leq C_{u_0}$ for satisfies \eqref{05} for any $t\in [0,T].$

	Now, we consider the uniqueness. Let  $n_1$ and $n_2$ be two solutions  in $E^p_T$ of \eqref{eq1} with the same data, we see that for $i=1,2$,  the function $N_i(t,\xi)=n(t,y_i(t,\xi))$ is a solution of
	\begin{align}
		&N_{it}(t,\xi)=n_{it}(t,y_i)+y_{it}(t,\xi)n_{ix}(t,y_i(t,\xi))=\tilde{Q}_i(t,\xi).\label{21}
	\end{align}
	Obviously, for sufficiently small $T>0,$ we have $N_i(t,y_i(t,\xi))\in L^{\infty}_T(W^{1,p}\cap W^{1,\infty}), y_i(t,\xi)-\xi\in L^{\infty}_T(W^{1,p}\cap W^{1,\infty})$ and $\frac{1}{2}\leq y_{i\xi}\leq C_{n_0}.$
	
	In order to get an estimate for $\|N_1(t,\xi)-N_2(t,\xi)\|_{W^{1,p}\cap W^{1,\infty}},$ we shall consider the estimate  for  $\|\tilde{Q}_{1}(t,\xi)-\tilde{Q}_{2 }(t,\xi)\|_{W^{1,p}\cap W^{1,\infty}}.$ For $\tilde{Q}_{1}(t,\xi)$ and $\tilde{Q}_{2 }(t,\xi),$ we have
	\begin{align}\label{q1}
		\tilde{Q}_{1}(t,\xi)-\tilde{Q}_{2}(t,\xi)&=N_1U_1(t,\xi)+\int_{\mathbb{R}}e^{-|y_1(\xi,t)-y_1(t,\eta)|}(U_1(t,\eta)N_{1x}(t,\eta)\frac{1}{y_\eta}-U_1(t,\eta)N_1(t,\eta))y_{1\eta} d\eta
		\notag \\
		&-\int_{\mathbb{R}}{\rm sgn}(y_1(t,\xi)-y_1(t,\eta))e^{-|y_1(\xi)-y_1(t,\eta)|}(U_1(t,\eta)N_{1x}(t,\eta)\frac{1}{y_{1\eta}}-U_1(t,\eta)N_1(t,\eta))y_{1\eta} d\eta
		\notag \\
		&+\lambda\int_{\mathbb{R}}{\rm sgn}(y_1(t,\xi)-y_1(t,\eta))e^{-|y_1(\xi)-y_1(t,\eta)|}{N_{1x}}^2(t,\eta)\frac{1}{y_{1\eta}}d\eta
		\notag \\
		&-\frac{1}{2}(-N_1^2-\int_{\mathbb{R}}e^{-|y_1(\xi)-\frac{1}{y_{1\eta}}|}{N_{1x}}^2(t,\eta)\frac{1}{y_\eta}d\eta+\int_{\mathbb{R}}e^{-|y_1(\xi)-y_1(t,\eta)|}{N_1}^2(t,\eta)y_{1\eta} d\eta
		\notag \\
		&+\int_{\mathbb{R}}{\rm sgn}(y_1(t,\xi)-y_1(t,\eta))e^{-|y_1(\xi)-y_1(t,\eta)|}{N_{1x}}^2(t,\eta)\frac{1}{y_{1\eta}}d\eta)
		\notag \\
		&+\frac{1}{2}(\partial_x+1)\int_{\mathbb{R}}\int_{\mathbb{R}}e^{-|y_1(\xi)-y_1(t,\zeta)|}e^{-|y_1(t,\zeta)-y_1(t,\eta)|}{N_{1x}}^2(t,\eta)\frac{y_{1\zeta}}{y_{1\eta}}d\zeta d\eta
		\notag \\
		&-N_2U_2(t,\xi)-\int_{\mathbb{R}}e^{-|y_2(\xi,t)-y_2(t,\eta)|}(U_2(t,\eta)N_{2x}(t,\eta)\frac{1}{y_{2\eta}}-U_2(t,\eta)N_2(t,\eta))y_{2\eta} d\eta
		\notag \\
		&+\int_{\mathbb{R}}{\rm sgn}(y_2(t,\xi)-y_2(t,\eta))e^{-|y_2(\xi)-y_2(t,\eta)|}(U_2(t,\eta)N_{2x}(t,\eta)\frac{1}{y_{2\eta}}-U_2(t,\eta)N_2(t,\eta))y_{2\eta} d\eta
		\notag \\
		&-\lambda\int_{\mathbb{R}}{\rm sgn}(y_2(t,\xi)-y_2(t,\eta))e^{-|y_2(\xi)-y_2(t,\eta)|}{N_{2x}}^2(t,\eta)\frac{1}{y_{2\eta}}d\eta
		\notag \\
		&+\frac{1}{2}(-N_2^2-\int_{\mathbb{R}}e^{-|y_2(\xi)-\frac{1}{y_{2\eta}}|}{N_{2x}}^2(t,\eta)\frac{1}{y_{2\eta}}d\eta+\int_{\mathbb{R}}e^{-|y_2(\xi)-y_2(t,\eta)|}{N_2}^2(t,\eta)y_{2\eta} d\eta
		\notag \\
		&+\int_{\mathbb{R}}{\rm sgn}(y_2(t,\xi)-y_2(t,\eta))e^{-|y_2(\xi)-y_2(t,\eta)|}{N_{2x}}^2(t,\eta)\frac{1}{y_{2\eta}}d\eta)
		\notag \\
		&-\frac{1}{2}(\partial_x+1)\int_{\mathbb{R}}\int_{\mathbb{R}}e^{-|y_2(\xi)-y_2(t,\zeta)|}e^{-|y_2(t,\zeta)-y_2(t,\eta)|}{N_{2x}}^2(t,\eta)\frac{y_{2\zeta}}{y_{2\eta}}d\zeta d\eta
	\end{align}

	Firstly, let us focaus on the term $G^{-1}(U_1N_{1x}-U_2N_{2x})$ which also can be written as
	
	$G^{-1}(N_1N_{1x}-N_2N_{2x})-N_1\partial_xG^{-1}N_1+N_2\partial_xN_2+N_1G^{-1}N_1-N_2G^{-1}N_2+\frac{1}{2\lambda}(N_1G^{-1}{N_{1x}}^2-N_2G^{-1}{N_{2x}}^2+N_1G^{-1}{N_1}^2-N_2G^{-1}{N_2}^2-N_1\partial_xG^{-1}{N_1}^2+N_2\partial_xG^{-1}{N_2}^2)$
	
	Then an estimate for $G^{-1}(N_1N_{1x}-N_2N_{2x})$ as the form of
	\begin{align}\label{q100}
		G^{-1}(N_1N_{1x}-N_2N_{2x})=&\int_{\mathbb{R}} sgn(y_1(\xi)-y_1(\eta))e^{-|y_1(\xi)-y_2(\eta)|}(N_1)^2y_{1\eta}d\eta
		\notag \\
		&-\int_{\mathbb{R}}sgn(y_2(\xi)-y_2(\eta))e^{-|y_2(\xi)-y_2(\eta)|}(N_2)^2y_{2\eta}d\eta,
	\end{align}
should be taken into consideration.

	Owing that  $y_{i}(t,\xi)(i=1,2)$ is
	monotonically increasing, then $sgn(y_i(\xi)-y_i(\eta))=sgn(\xi-\eta),$  which implies \eqref{q100} can be rewritten as
	\begin{align}\label{q2}
		G^{-1}(N_1N_{1x}-N_2N_{2x})=&\int_{\mathbb{R}} sgn(\xi-\eta)[e^{-|y_1(\xi)-y_1(\eta)|}N_1^2y_{1\eta}-e^{-|y_2(\xi)-y_2(\eta)|}N_2^2y_{2\eta}]
		\notag\\
		&=\int_{\mathbb{R}} sgn(\xi-\eta)[e^{-|y_1(\xi)-y_1(\eta)|}N_1^2y_{1\eta}-e^{-|y_2(\xi)-y_2(\eta)|}N_1^2y_{2\eta}
		\notag \\
		&+e^{-|y_2(\xi)-y_2(\eta)|}N_1^2y_{1\eta}-e^{-|y_2(\xi)-y_2(\eta)|}N_2^2y_{1\eta}] d\eta
		\notag \\
		&=\int_{\mathbb{R}}sgn(\xi-\eta)(e^{-|y_1(\xi)-y_1(\eta)|}-e^{-|y_2(\xi)-y_2(\eta)|})N_1^2y_{1\eta}
		\notag \\
		&+sgn(\xi-\eta)e^{-|y_2(\xi)-y_2(\eta)|}(N_1^2y_{1\eta}-N_2^2y_{2\eta})d\eta
		\notag \\
		&=I_1+I_2.
	\end{align}
	
	If $\xi>\eta(or \ \xi<\eta),$ then $y_i(\xi)>y_i(\eta)~(or~y_i(\xi)<y_i(\eta)).$ Thus, we have
	\begin{align}\label{23}
		I_1=&-\int_{\xi}^{\infty}(e^{-|y_1(\xi)-y_1(\eta)|}-e^{-|y_2(\xi)-y_2(\eta)|})N_1^2y_{1\eta}d\eta 
		\notag\\
		&~+\int_{-\infty}^{\xi}((e^{-|y_1(\xi)-y_1(\eta)|}-e^{-|y_2(\xi)-y_2(\eta)|})N_1^2y_{1\eta}d\eta  
		\notag\\
		=&-\int_{\xi}^{\infty}(e^{\int_{0}^{t} U_1(\xi)-U_1(\eta)d\tau}-
		e^{\int_{0}^{t} U_2(\xi)-U_2(\eta)d\tau})N_1^2 y_{1\eta}d\eta
		\notag\\
		&+\int_{\infty}^{\xi}(e^{\int_{0}^{t} U_1(\xi)-U_1(\eta)d\tau}-
		e^{\int_{0}^{t} U_2(\xi)-U_2(\eta)d\tau})N_1^2 y_{1\eta}d\eta
		\notag \\
		\leq&C\|U_1-U_2\|_{L^{\infty}} (1_{\geq 0}e^{-|\cdot|}\ast \frac{1}{2}\partial_x(N_1)^2+
		1_{\leq 0}e^{-|\cdot|}\ast \frac{1}{2}\partial_x(N_1)^2)
	\end{align}
	Likewise, we get
	\begin{align}\label{1000}
		I_2=&\int_{\mathbb{R}}e^{-|y_2(\xi)-y_2(\eta)|}(N_1N_{1x}-N_2N_{2x})d\eta
		\notag \\
		&=\int_{-\infty}^{\xi}e^{y_2(\eta)-y_2(\xi)}(N_1N_{1x}-N_2N_{2x})d\eta
		\notag \\
		&~~~+\int_{\xi}^{+\infty}e^{y_2(\xi)-y_2(\eta)}(N_1N_{1x}-N_2N_{2x})d\eta
		\notag \\
		&\leq C(1_{\geq 0}e^{-|\cdot|}\ast\frac{1}{2}(N_1^2-N_2^2)_x
		+1_{\leq 0}e^{-|\cdot|}\ast\frac{1}{2}(N_1^2-N_2^2)_x)
	\end{align}
	Combining \eqref{23} and \eqref{1000},then
	\begin{align}
		&\|G^{-1}(N_1N_{1x}-N_2N_{2x})\|_{L^{\infty}}\leq C(\|U_1-U_2\|_{L^{\infty}} \|(1_{\geq 0}e^{-|\cdot|}\ast \frac{1}{2}\partial_x(N_1)^2+
		1_{\leq 0}e^{-|\cdot|}\ast \frac{1}{2}\partial_x(N_1)^2)\|_{L^{\infty}}+
		\notag\\
		&\|1_{\geq 0}e^{-|\cdot|}\ast\frac{1}{2}(N_1^2-N_2^2)_x
		+1_{\leq 0}e^{-|\cdot|}\ast\frac{1}{2}(N_1^2-N_2^2)_x\|_{L^{\infty}})
		\notag\\
		&\leq C(\|U_1-U_2\|_{L^{\infty}} \frac{1}{2}\|(1_{\geq 0}e^{-|\cdot|}\|_{L^1} \|\partial_x(N_1)\|_{L^{\infty}}
		\|N_1\|_{L^{\infty}}+
		\|1_{\leq 0}e^{-|\cdot|}\|_{L^1} \frac{1}{2}\|\partial_x(N_1)\|_{L^{\infty}}
		\|N_1\|_{L^{\infty}}
		\notag\\
		&+\|1_{\geq 0}e^{-|\cdot|}\|_{L^1} (\|N_1-N_2\|_{L^{\infty}}\|N_{1x}\|_{L^{\infty}}+\|N_{1x}-N_{2x}\|_{L^{\infty}}\|N_{2}\|_{L^{\infty}})
		\notag\\
		&+\|1_{\leq 0}e^{-|\cdot|}\|_{L^1} (\|N_1-N_2\|_{L^{\infty}}\|N_{1x}\|_{L^{\infty}}+\|N_{1x}-N_{2x}\|_{L^{\infty}}\|N_{2}\|_{L^{\infty}}))
		\notag\\
		&\leq C\|U_1-U_2\|_{L^{\infty}}(C+\|N_1-N_2\|_{L^{\infty}})
		\notag\\
		&\|G^{-1}(N_1N_{1x}-N_2N_{2x})\|_{L^p}\leq C(\|U_1-U_2\|_{L^{\infty}} \|(1_{\geq 0}e^{-|\cdot|}\ast \frac{1}{2}\partial_x(N_1)^2+
		1_{\leq 0}e^{-|\cdot|}\ast \frac{1}{2}\partial_x(N_1)^2)\|_{L^p}+
		\notag\\
		&\|1_{\geq 0}e^{-|\cdot|}\ast\frac{1}{2}(N_1^2-N_2^2)_x
		+1_{\leq 0}e^{-|\cdot|}\ast\frac{1}{2}(N_1^2-N_2^2)_x\|_{L^p})
		\notag\\
		&\leq C(\|U_1-U_2\|_{L^{\infty}} \frac{1}{2}\|(1_{\geq 0}e^{-|\cdot|}\|_{L^1} \|\partial_x(N_1)\|_{L^p}
		\|N_1\|_{L^{\infty}}+
		\|1_{\leq 0}e^{-|\cdot|}\|_{L^1} \frac{1}{2}\|\partial_x(N_1)\|_{L^p}
		\|N_1\|_{L^{\infty}}
		\notag\\
		&+\|1_{\geq 0}e^{-|\cdot|}\|_{L^1} (\|N_1-N_2\|_{L^{p}}\|N_{1x}\|_{L^{\infty}}+\|N_{1x}-N_{2x}\|_{L^{p}}\|N_{2}\|_{L^{\infty}})
		\notag\\
		&+\|1_{\leq 0}e^{-|\cdot|}\|_{L^1} (\|N_1-N_2\|_{L^{p}}\|N_{1x}\|_{L^{\infty}}+\|N_{1x}-N_{2x}\|_{L^{p}}\|N_{2}\|_{L^{\infty}}))
		\notag\\
		&\leq C\|U_1-U_2\|_{L^{\infty}}(C+\|N_1-N_2\|_{L^{p}})
	\end{align}
	Similarly, we can obtain
	\begin{align}
		&\partial_xG^{-1}N_{1\xi}^2-\partial_xG^{-1}N_{2\xi}^2=\int_{\mathbb{R}} sgn(y_1(\xi)-x)e^{-|y_1(\xi)-x|}n_{1x}^2dx-\int_{\mathbb{R}} sgn(y_2(\xi)-x)e^{-|y_2(\xi)-x|}n_{2x}^2dx
		\notag \\
		&=\int_{\mathbb{R}} sgn(y_1(\xi)-y_{1\eta}(\eta))e^{-|y_1(\xi)-y_1(\eta)|}N_{1\eta}^2\frac{1}{y_{1\eta}}d\eta-\int_{\mathbb{R}} sgn(y_2(\xi)-y_{2\eta}(\eta))e^{-|y_2(\xi)-y_2(\eta)|}N_{2\eta}^2\frac{1}{y_{2\eta}}d\eta
		\notag \\
		&=\int_{\mathbb{R}} sgn(\xi-\eta)e^{-|y_1(\xi)-y_1(\eta)|}N_{1\eta}^2\frac{1}{y_{1\eta}}d\eta-\int_{\mathbb{R}} sgn(\xi-\eta)e^{-|y_1(\xi)-y_1(\eta)|}N_{2\eta}^2\frac{1}{y_{2\eta}}d\eta
		\notag \\
		&~~~+\int_{\mathbb{R}} sgn(\xi-\eta)e^{-|y_1(\xi)-y_1(\eta)|}N_{2\eta}^2\frac{1}{y_{2\eta}}d\eta
		-\int_{\mathbb{R}} sgn(\xi-\eta)e^{-|y_2(\xi)-y_2(\eta)|}N_{2\eta}^2\frac{1}{y_{2\eta}}d\eta
		\notag \\
		&=\int_{\mathbb{R}}sgn(\xi-\eta)[e^{-|y_1(\xi)-y_1(\eta)|}-e^{-|y_1(\xi)-y_1(\eta)|}]
		\frac{{N_{1\eta}}^2}{y_{1\eta}}+e^{-|y_1(\xi)-y_1(\eta)|}(\frac{{N_{1\eta}}^2}{y_{1\eta}}-\frac{{N_{2\eta}}^2}{y_{2\eta}})d\eta
		\notag \\
		&\leq C\|N_1-N_2\|_{L^{\infty}}[1_{\leq 0}e^{-|\cdot|}\ast \frac{{N_{1\eta}}^2}{y_{1\eta}}+1_{\geq 0}e^{-|\cdot|}\ast \frac{{N_{1\eta}}^2}{y_{1\eta}}]
		\notag \\
		&+C[1_{\leq 0}e^{-|\cdot|}\ast (|N_{1\eta} -N_{2\eta}|+|y_{1\eta} -{y_{2\eta}}|)+1_{\geq 0}e^{-|\cdot|}\ast (|N_{1\eta} -N_{2\eta}|+|y_{1\eta} -y_{2\eta}|)].
	\end{align}
	And the last term we analyze is $G^{-1}G^{-1}[N_{1\xi}^2-N_{2\xi}^2]$, which is as the form of 
	\begin{align}
		&G^{-1}G^{-1}[N_{1x}^2-N_{2x}^2]=
		\notag \\
		&\int_{\mathbb{R}}\int_{\mathbb{R}}[e^{-|y_1(\xi)-y_1(\zeta)|}e^{-|y_1(\zeta)-x|}{N_{1\xi}}^2y_{1\zeta}]-[e^{-|y_2(\xi)-y_2(\zeta)|}e^{-|y_2(\zeta)-x|}{N_{2\xi}}^2(\eta)y_{2\zeta}]d\zeta dx
		\notag \\
		&=\int_{\mathbb{R}}\int_{\mathbb{R}}[e^{-|y_1(\xi)-y_1(\zeta)|}e^{-|y_1(\zeta)-y_1(\eta)|}{N_{1\xi}}^2\frac{y_{1\zeta}}{y_{1\eta}}]-[e^{-|y_2(\xi)-y_2(\zeta)|}e^{-|y_2(\zeta)-y_2(\eta)|}{N_{2\xi}}^2(\eta)\frac{y_{2\zeta}}{y_{2\eta}}]d\zeta d\eta
		\notag \\
		&=\int_{\mathbb{R}}\int_{\mathbb{R}}[e^{-|y_1(\xi)-y_1(\zeta)|}e^{-|y_1(\zeta)-y_1(\eta)|}{N_{1\xi}}^2\frac{y_{1\zeta}}{y_{1\eta}}]-[e^{-|y_1(\xi)-y_1(\zeta)|}e^{-|y_2(\zeta)-y_2(\eta)|}{N_{1\xi}}^2\frac{y_{1\zeta}}{y_{1\eta}}]
		\notag \\
		&~~~+[e^{-|y_1(\xi)-y_1(\zeta)|}e^{-|y_2(\zeta)-y_2(\eta)|}{N_{1\xi}}^2\frac{y_{1\zeta}}{y_{1\eta}}]
		-[e^{-|y_1(\xi)-y_1(\zeta)|}e^{-|y_2(\zeta)-y_2(\eta)|}{N_{2\xi}}^2(\eta)\frac{y_{2\zeta}}{y_{2\eta}}]
		\notag \\
		&~~~+[e^{-|y_1(\xi)-y_1(\zeta)|}e^{-|y_2(\zeta)-y_2(\eta)|}{N_{2\xi}}^2\frac{y_{2\zeta}}{y_{2\eta}}]-[e^{-|y_2(\xi)-y_2(\zeta)|}e^{-|y_2(\zeta)-y_2(\eta)|}{N_{2\xi}}^2(\eta)\frac{y_{2\zeta}}{y_{2\eta}}]d\zeta d\eta
		\notag \\
		&\leq C\int_{\mathbb{R}}\|N_1-N_2\|_{L^{\infty}}[e^{-|y_1(\xi)-y_1(\zeta)|}1_{\geq 0}e^{-|\cdot|}\ast ({N_{1\xi}}^2\frac{y_{1\zeta}}{y_{1\eta}})]-[e^{-|y_1(\xi)-y_1(\zeta)|}1_{\textless 0}e^{-|\cdot|}\ast ({N_{1\xi}}^2\frac{y_{1\zeta}}{y_{1\eta}})]d\zeta
		\notag \\
		&~~~+1_{\leq 0}e^{-|\cdot|}\ast e^{-|\cdot|}\ast(|N_{1\xi}-N_{2\xi}|+|y_{1\xi}-y_{2\xi}|)
		+1_{\textgreater 0}e^{-|\cdot|}\ast e^{-|\cdot|}\ast(|N_{1\xi}-N_{2\xi}|+|y_{1\xi}-y_{2\xi}|)
		\notag \\
		&~~~+C\int_{\mathbb{R}}\|N_1-N_2\|_{L^{\infty}}[1_{\geq 0}e^{-|\cdot|}\ast (e^{-|\cdot-y_2(\zeta)|}{N_{2\xi}}^2\frac{y_{2\zeta}}{y_{2\eta}})]-[1_{\textless 0}e^{-|\cdot|}\ast (e^{-|\cdot-y_1(\zeta)|}{N_{2\xi}}^2\frac{y_{2\zeta}}{y_{2\eta}})]d\eta
		\notag \\
		&\leq C\|N_1-N_2\|_{L^{\infty}}[1_{\geq 0}e^{-|\cdot|}\ast e^{-|\cdot|}\ast ({N_{1\xi}}^2)]-[1_{\textless 0}e^{-|\cdot|}\ast e^{-|\cdot|}\ast ({N_{1\xi}}^2)]
		\notag \\
		&~~~+1_{\leq 0}e^{-|\cdot|}\ast e^{-|\cdot|}\ast(|N_{1\xi}-N_{2\xi}|+|y_{1\xi}-y_{2\xi}|)
		+1_{\textgreater 0}e^{-|\cdot|}\ast e^{-|\cdot|}\ast(|N_{1\xi}-N_{2\xi}|+|y_{1\xi}-y_{2\xi}|).
	\end{align}
	So we have 
	\begin{align}
		\\
		&\|G^{-1}G^{-1}[N_{1x}^2-N_{2x}^2]\|_{L^{\infty}}\leq
		\notag\\ &C\|N_1-N_2\|_{L^{\infty}}(\|e^{-|\cdot|}\|_{L^1}^2\|N_{1\xi}\|_{L^{\infty}}^2+\|e^{-|\cdot|}\|_{L^1}^2(\|N_{1\xi}^2-N_{2\xi}^2\|_{L^{\infty}}+\|y_{1\xi}^2-y_{2\xi}^2\|_{L^{\infty}}))
		\notag\\
		&\|G^{-1}G^{-1}[N_{1x}^2-N_{2x}^2]\|_{L^p}\leq
		\notag\\ &C\|N_1-N_2\|_{L^p}(\|e^{-|\cdot|}\|_{L^1}^2\|N_{1\xi}\|_{L^p}^2+\|e^{-|\cdot|}\|_{L^1}^2(\|N_{1\xi}^2-N_{2\xi}^2\|_{L^p}+\|y_{1\xi}^2-y_{2\xi}^2\|_{L^p}))
	\end{align}

	Then the proof of the other terms will follow the same line,the details should be omitted.And it is easy to check that 
	\begin{align}
		&\|U_1-U_2\|_{L^{\infty}\bigcap L^{P}}\leq \|N_1-N_2\|_{W^{1,\infty}\bigcap W^{1,P}}+
		\|y_1-y_2\|_{W^{1,\infty}\bigcap W^{1,P}}
		\notag \\
		&\|N_1U_1-N_2U_2\|_{L^{\infty}\bigcap L^{P}}\leq \|N_1-N_2\|_{L^{\infty}\bigcap L^{P}}+
		\|U_1-U_2\|_{L^{\infty}\bigcap L^{P}}
		\notag \\
		&\leq C\|N_1-N_2\|_{L^{\infty}\bigcap L^{P}}
	\end{align}
	Acroding to the analysis above, it's not difficult to check that
	\begin{align}\label{24}
		\|\tilde{Q}_{1}(t,\xi)-\tilde{Q}_{2}(t,\xi)\|_{L^{\infty}\cap L^p}\leq C(\|N_1-N_2\|_{L^{\infty}\cap L^p}+\|N_{1\eta}-N_{2\eta}\|_{L^{\infty}\cap L^p}+\|y_{1\eta}-y_{2\eta}\|_{L^{\infty}\cap L^p}).
	\end{align}
	In the same way, we have
	\begin{align}\label{25}
		\|\tilde{Q}_{1\xi}(t,\xi)-\tilde{Q}_{2\xi}(t,\xi)\|_{L^{\infty}\cap L^p}\leq C(\|N_1-N_2\|_{L^{\infty}\cap L^p}+\|N_{1\eta}-N_{2\eta}\|_{L^{\infty}\cap L^p}+\|y_{1\eta}-y_{2\eta}\|_{L^{\infty}\cap L^p}).
	\end{align}
	Combining \eqref{24} and \eqref{25}, we can find that
	\begin{align}\label{26}
		\|\tilde{Q}_{1}(t,\xi)-\tilde{Q}_{2}(t,\xi)\|_{W^{1,p}\cap W^{1,\infty}}\leq C(\|N_1-N_2\|_{W^{1,p}\cap W^{1,\infty}}+\|y_1-y_2\|_{W^{1,p}\cap W^{1,\infty}}).
	\end{align}
	Moreover,
	\begin{align}\label{27}
		&\|N_1-N_2\|_{W^{1,p}\cap W^{1,\infty}}+\|y_1-y_2\|_{W^{1,p}\cap W^{1,\infty}}\notag\\&\leq C (\|N_1(0)-N_2(0)\|_{W^{1,p}\cap W^{1,\infty}}+\|y_1(0)-y_2(0)\|_{W^{1,p}\cap W^{1,\infty}})\notag\\
		&~~+C\int_0^T(\|N_1-N_2\|_{W^{1,p}\cap W^{1,\infty}}+\|y_1-y_2\|_{W^{1,p}\cap W^{1,\infty}})dt.
	\end{align}
	It means that
	\begin{align}
		&\|N_1-N_2\|_{W^{1,p}\cap W^{1,\infty}}+\|y_1-y_2\|_{W^{1,p}\cap W^{1,\infty}}\notag\\&\leq  e^{CT}(\|N_1(0)-N_2(0)\|_{W^{1,p}\cap W^{1,\infty}}+\|y_1(0)-y_2(0)\|_{W^{1,p}\cap W^{1,\infty}})\notag\\&\leq e^{CT}(\|N_1(0)-N_2(0)\|_{W^{1,p}\cap W^{1,\infty}}+0)\notag\\&\leq e^{CT} \|n^0_1-n^0_2\|_{B^{1+\frac{1}{p}}_{p,1}},
	\end{align}
	where we use the fact that  $y_1(0)=y_2(0)=\xi.$ It follows that
	\begin{align*}
		\|n_1-n_2\|_{L^p}&\leq C \|n_1 \circ y_1-	n_2 \circ y_1\|_{L^p}\notag\\&\leq C\|n_1 \circ y_1-n_2 \circ y_2+n_2 \circ y_2-n_2 \circ y_1\|_{L^p} \notag\\&\leq C\|N_1-N_2\|_{L^p}+C\|n_{2x}\|_{L^{\infty}}\|y_1-y_2\|_{L^p}\notag\\&\leq C \|n^0_1-n^0_2\|_{B^{1+\frac{1}{p}}_{p,1}}.
	\end{align*}
	And the key to the uniqueness is
	\begin{align}
		&\|\gamma_1-\gamma_2\|_{B^{\frac{1}{p}}_{p,1}}\leq C\|(\partial_x -1)n_1-(\partial_x -1)n_2\|_{B^{\frac{1}{p}}_{p,1}}
		\notag \\
		&\leq C\|n_1-n_2\|_{B^{1+\frac{1}{p}}_{p,1}}
	\end{align}
	The embedding $L^p\hookrightarrow  B^{0}_{p,\infty}$ ensures that
	\begin{align*}
		\|n_1-n_2\|_{B^{0}_{p,\infty}}\leq C\|n_1-n_2\|_{L^p}\leq C\|n^0_1-n^0_2\|_{B^{1+\frac{1}{p}}_{p,1}}.
	\end{align*}
	If $\gamma^0_1=\gamma^0_2,$ the uniqueness of the solution is proved.\\
	
	\textbf{Step 3: The continuous dependence.}
	
	Let $n_n,n_{\infty}$ be the solutions of \eqref{eq1} with the initial data $n_{0n},n_{0\infty}$ respectively and let $n_{0n}$ tend to $n_{0\infty}$ in $B^{1+\frac1 p }_{p,1}.$ Combining \textbf{Step 1} and \textbf{Step 2}, due to the uniform bound of $\gamma$ we get $n_n,~n_{\infty}$ which are uniformly bounded in $L^{\infty}_T(B^{1+\frac1 p }_{p,1})$ and
	\begin{align}\label{4b}
		\|(n_m-n_{\infty})(t)\|_{B^{0 }_{p,\infty}}\leq C\|n^m_0-n^{\infty}_0\|_{B^{1+\frac1 p }_{p,1}}.
	\end{align}
	This means that  $n_m$ tends to $n_{\infty}$ in $C([0,T],B^{0 }_{p,\infty}).$ Using interpolation, we see that $n_m\rightarrow n_{\infty}$ in $C([0,T],B^{1+\frac{1}{p}-\varepsilon}_{p,1})$ for any $\varepsilon >0$. If $\varepsilon =1$,  we get
	\begin{align}\label{a0}
		n_m\rightarrow n_{\infty} \ in \ C([0,T],B^{\frac{1}{P}}_{p,1}).
	\end{align}
	Combining \eqref{4b} and \eqref{a0}, we only need to prove that  $\partial_xn_m\rightarrow\partial_xn_{\infty}$ in $C([0,T],{B^{\frac1 p }_{p,1}}).$ Letting $v_m=\partial_xn_m,$  we split $v_m=z_m+w_m$  with $(z_m,w_m)$ satisfying
	\begin{equation*}
		\left\{\begin{aligned}
			&\partial_tz_m-u_m\partial_xz_m=\partial_xQ_{\infty}+\partial_x u_{\infty}\partial_xn_{\infty}, \\
			&z_m|_{t=0}=\partial_xu_{0{\infty}}. \\
		\end{aligned}\right.
	\end{equation*}
	and
	\begin{equation*}
		\left\{\begin{aligned}
			&\partial_tw_m-u_m\partial_xw_m=\partial_xQ_{m}-\partial_xQ_{\infty}+
			-(\partial_xu_m\partial_xn_m-\partial_xu_\infty\partial_xn_\infty), \\
			&w_m|_{t=0}=\partial_xu_{0m}-\partial_xu_{0{\infty}},
		\end{aligned}\right.
	\end{equation*}
	Making use of the fact that $(u_m)_{m\in \mathbb{N}} $ and $u_{\infty} $ are bounded in $B^{1+\frac{1}{p}}_{p,1},$ we have
	\begin{align}\label{eqc5}
		\|\partial_x(u_m)\|_{B^{\frac{1}{p}}_{p,1}}&\leq C\|(u_m)\|_{B^{1+\frac{1}{p}}_{p,1}},\\
		\|\partial_x u_{m}\partial_x n_{m}-\partial_x u_{\infty}\partial_x n_{\infty}\|_{B^{\frac{1}{p}}_{p,1}}&\leq C \|u_m-u_{\infty}\|_{B^{1+\frac{1}{p}}_{p,1}}\|n_m-n_{\infty}\|_{B^{1+\frac{1}{p}}_{p,1}}
	\end{align}
	and
	\begin{align}
		\|\partial_x Q_n-\partial_xQ_{\infty}\|_{B^{\frac{1}{p}}_{p,1}}& \leq C \|n_m-n_{\infty}\|_{B^{1+\frac{1}{p}}_{p,1}} \notag\\
		&\leq C \Big(\|n_m-n_{\infty}\|_{B^{\frac{1}{p}}_{p,1}}+\|\partial_xn_m-\partial_xn_{\infty}\|_{B^{\frac{1}{p}}_{p,1}}
		\Big).\label{cau}
	\end{align}
	Then, using the inequalities \eqref{eqc5}-\eqref{cau}, we get, for all $n\in \mathbb{N},$
	\begin{align*}
		\partial_t\|w_m	(t)\|_{B^{\frac{1}{p}}_{p,1}}&\leq C\Big(\|n_m-n_{\infty}\|_{B^{\frac{1}{p}}_{p,1}}+\|\partial_xn_m
		-\partial_xn_{\infty}\|_{B^{\frac{1}{p}}_{p,1}}\Big)\notag\\
		&\leq C\Big(\|n_m-n_{\infty}\|_{B^{\frac{1}{p}}_{p,1}}+\|z_m
		-z_{\infty}\|_{B^{\frac{1}{p}}_{p,1}}+\|w_m
		(t)\|_{B^{\frac{1}{p}}_{p,1}}\Big),
	\end{align*}
	from which it follows
	\begin{align}\label{cau1}
		\|w_m	(t)\|_{B^{\frac{1}{p}}_{p,1}}& \leq Ce^{Ct} \Big(\|v^0_m-v^0_{\infty}\|_{B^{\frac{1}{p}}_{p,1}}+\int_0^t e^{-Ct'}  (\|n_m-n_{\infty}\|_{B^{\frac{1}{p}}_{p,1}}+\|z_m-z_{\infty}\|_{B^{\frac{1}{p}}_{p,1}}+\|w_m
		\|_{B^{\frac{1}{p}}_{p,1}})	(t')dt'\Big).
	\end{align}
	Remember that $v_{0m}$ tends to $v_{0\infty}$ in $B^{\frac{1}{p}}_{p,1}, n_m$ tends to $n_{\infty}$ in $C([0,T]; B^{\frac{1}{p}}_{p,1}),$  By  Lemma \ref{cont1}, we have 	$z_m	(t) \rightarrow z_{\infty}	(t) \  in \  C([0,T];B^{\frac{1}{p}}_{p,1}).$ Hence we obtain that $w_m$ tends to $0$ in $C([0,T]; B^{\frac{1}{p}}_{p,1}).$ Therefore, applying Lemmas \ref{existence} , \ref{priori estimate} and $w_{\infty}=0$, we get $w_m$ tends to $w_{\infty}$ in $C([0,T]; B^{\frac{1}{p}}_{p,1}).$

	Finally, we conclude that
	\begin{align}
		\|v_m-v_{\infty}\|_{L^{\infty}_T (B^{\frac{1}{p}}_{p,1})}&\leq \|z_m-z_{\infty}\|_{{L^{\infty}_T (B^{\frac{1}{p}}_{p,1})}}+\|w_m-w_{\infty}\|_{{L^{\infty}_T (B^{\frac{1}{p}}_{p,1})}}\notag\\
		&\leq \|z_m-z_{\infty}\|_{{L^{\infty}_T (B^{\frac{1}{p}}_{p,1})}}+\|w_m\|_{{L^{\infty}_T (B^{\frac{1}{p}}_{p,1})}},\label{ok}
	\end{align}
	which implies that
	\begin{align*}
		\partial_xn_m \rightarrow \partial_xn_{\infty}  \ \ \ in \ \ \ C([0,T]; B^{\frac{1}{p}}_{p,1}).
	\end{align*}
	Combining  \textbf{Step 1} to \textbf{Step 3}, we complete  the proof of Theorem \ref{the1}.
\end{proof}
\section{Persistence properties}
\par\
In this section,we study if a classical solution $\gamma$ of $ \eqref{eq0}$ starts out having compact support, then this property will be inherited by $\gamma$ at all times $t\in[0, T).$
\begin{theo}\label{the2}
	Let $\gamma_0$ be in $H^s,s\geq 4$ such that $\gamma_0$ has compact support.If T=T($\gamma_0$)$\textgreater 0$is the maximal existence time of the unique solution $\gamma(x,t)$ to $\eqref{eq0}$ with initial data $\gamma_0$, then for any t$\in [0,T)$,$\gamma$has compact support.
\end{theo}
\begin{proof}
	For the first equation of \eqref{eq0} and by Lemma 2.11, we have 
	\begin{align}
		\gamma_t=&\lambda(\gamma_x+\frac{\gamma^2}{2\lambda}-\gamma_x-\frac{1}{\lambda}v_x\gamma-\frac{1}{\lambda}v\gamma_x)
		\notag\\
		&=(\lambda \gamma_x+\frac{\gamma^2}{2}-\lambda\gamma_x-v_x\gamma-v\gamma_x)
		\notag\\
		&=\frac{\gamma^2}{2}-v_x\gamma-v\gamma_x.
	\end{align}
	Then we consider the initial problem $\eqref{eq102}$
	\begin{equation}\label{eq102}
		\left\{\begin{array}{l}
			y_t=v(t,y),t\in [0,T) , \\
			
			y(0,x)=x,x\in\mathbb{R}.
		\end{array}\right.
	\end{equation}
	And we will have $y_xt=v_xy_x$
	\begin{align}
		\frac{d}{dt}(\gamma y_x)=&\frac{d}{dt}\gamma y_x+\gamma_x\frac{d}{dt}y y_x+\gamma\frac{d}{dt}yx
		\notag\\
		&=\frac{d}{dt}\gamma y_x+\gamma_xv y_x+\gamma v_xy_x
		\notag\\
		&=\frac{\gamma^2}{2}y_x .
	\end{align}
	Then we can find $\gamma y_x=(\gamma_0 y_x(0))exp(\int\frac{\gamma}{2}dt^{\prime})$.
	So if$\gamma_0$has compact support,then $\gamma$has compact suppport.
\end{proof}
\smallskip
\noindent\textbf{Acknowledgments} This work was
partially supported by the National Natural Science Foundation of China (No.12171493).

\noindent\textbf{Data Availability.}
The data that support the findings of this study are available on citation. The data that support the findings of this study are also available
from the corresponding author upon reasonable request.

	\phantomsection
	\addcontentsline{toc}{section}{\refname}
	%Ìí¼Ó²Î¿¼ÎÄÏ×µ½ÊéÇ©£¬ºê°ü hyperref
	\bibliographystyle{abbrv} %plain ,%alpha, %abbrv
	\bibliography{Feneref}
\end{document}